\documentclass[12pt, a4paper, english]{amsart}
\usepackage{amsmath} % utilities for mathematics
\usepackage{amsthm} % utilities for theorem environment
\usepackage{amssymb} % loads mathematical fonts
\usepackage{amscd} % utilities for commutative diagrams
\usepackage{array} % core package
\usepackage[mathscr]{eucal} % loads a font accessible by \mathscr{} 
\usepackage{hyperref} % permits navigating on the pdf
\usepackage{caption} %
\usepackage{graphics,graphicx} % permits putting images
\usepackage{tikz,tikz-cd} % tool for diagrams
\usepackage{enumerate} % permits enumerating using letters or other symbols

\usepackage[margin=2.55cm]{geometry}
\usepackage{verbatim}
\usepackage{xcolor}
\usepackage{microtype}

\usetikzlibrary{automata}

\newcommand\RR{{\mathbb R}}
\newcommand\ZZ{{\mathbb Z}}
\newcommand\ii{{\operatorname{i}}}

\theoremstyle{plain}
\newtheorem{teo}{Theorem}[section]
\newtheorem{theorem}[teo]{Theorem}
\newtheorem{corollary}[teo]{Corollary}
\newtheorem{lemma}[teo]{Lemma}
\newtheorem{conj}[teo]{Conjecture}

\newtheorem{proposition}[teo]{Proposition}

\theoremstyle{definition}

\newtheorem{example}[teo]{Example}

\newtheorem{remark}[teo]{Remark}
\newtheorem{rem}[teo]{Remark}

\DeclareMathOperator{\ehr}{ehr}
\DeclareMathOperator{\rk}{rk}

\title{Ehrhart polynomials of rank two matroids}

\author{Luis Ferroni}
\author{Katharina Jochemko}
\author{Benjamin Schr\"oter}

\address{
  Department of Mathematics, KTH Royal Institute of Technology, Stockholm, Sweden
}
\email{\{ferroni,jochemko,schrot\}@kth.se}


\allowdisplaybreaks

\begin{document}

\begin{abstract}
Over a decade ago De Loera, Haws and K\"oppe conjectured that Ehrhart polynomials of matroid polytopes have only positive coefficients and that the coefficients of the corresponding $h^*$-polynomials form a unimodal sequence. The first of these intensively studied conjectures has recently been disproved by the first author who gave counterexamples in all ranks greater or equal to three. In this article we complete the picture by showing that Ehrhart polynomials of matroids of lower rank have indeed only positive coefficients. Moreover, we show that they are coefficient-wise bounded by the Ehrhart polynomials of minimal and uniform matroids. We furthermore address the second conjecture by  proving that $h^*$-polynomials of matroid polytopes of sparse paving matroids of rank two are real-rooted and therefore have log-concave and unimodal coefficients. In particular, this shows that the $h^*$-polynomial of the second hypersimplex is real-rooted thereby strengthening a result of De Loera, Haws and K\"oppe.
\end{abstract}

\maketitle

\section{Introduction}
\noindent
A \textbf{lattice polytope} in $\mathbb{R}^n$ is defined as the convex hull of finitely many vectors in the integer lattice $\mathbb{Z}^n$. A fundamental theorem by Ehrhart~\cite{Ehrhart} states that for any lattice polytope $\mathscr{P}$ the number of lattice points in the $t$-th dilate $t\mathscr{P}$ is given by a polynomial $\ehr (\mathscr{P},t)$ for all integers $t\geq 0$. The polynomial $\ehr (\mathscr{P},t)$, called the \textbf{Ehrhart polynomial} of the polytope $\mathscr{P}$, encodes geometric and combinatorial information about the lattice polytope such as its dimension and its volume which are equal to the degree and the leading coefficient, respectively. Characterizing Ehrhart polynomials, including finding interpretations for their coefficients, is an intensively studied question that remains widely open. One difficulty is the fact that the coefficients of the Ehrhart polynomial can be negative in general. Lattice polytopes whose Ehrhart polynomials have only positive (or nonnegative) coefficients are therefore of particular interest. Such polytopes are called \textbf{Ehrhart positive}. For further reading on Ehrhart positivity we recommend~\cite{EhrhartPositivity}.

This article is concerned with Ehrhart polynomials of matroid polytopes. Given a matroid $M$ on the ground set $E=\{1,\ldots,n\}$ with set of bases $\mathscr{B}\subseteq 2^E$ the \textbf{matroid (base) polytope} $\mathscr{P}(M)$ of $M$ is defined as
    \[ \mathscr{P}(M) := \operatorname{conv}\{e_B : B\in \mathscr{B}\}\subseteq \mathbb{R}^n\]
where $e_B := \sum_{i\in B} e_i$ is the indicator vector of the basis $B\in \mathscr{B}$ and $e_1,\ldots,e_n$ denotes the canonical basis of $\mathbb{R}^n$. 

Over a decade ago De Loera, Haws and K\"oppe conjectured that matroid polytopes are Ehrhart positive~\cite{DeLoera}. This conjecture together with companion conjectures has attracted considerable attention in the recent years \cite{GeneralizedPerm,CastilloTodd,ferroni-hypersimplices,jochemko2019generalized,knauer18}. Castillo and Liu~\cite{GeneralizedPerm} conjectured Ehrhart positivity for the larger class of generalized permutohedra (also known as polymatroids). In~\cite{ferroni2020ehrhart} the first author conjectured that the coefficients of matroid polytopes are not only positive but are moreover coefficient-wise bounded from below and above by the minimal matroid and the uniform matroid, respectively. Recently, the first author disproved these conjectures simultaneously by providing examples with negative coefficients of matroids whose rank ranges between three and corank three~\cite{ferroni2021matroids}. 

In this article we complete the picture by proving Ehrhart positivity for all matroids of rank $2$ or equivalently corank $2$. Matroid polytopes of rank~$1$ or corank~$1$ are unimodular simplices and therefore Ehrhart positive. 

One of our main results is the following concrete formula for Ehrhart polynomials of matroid polytopes of rank $2$ matroids that generalizes a formula for hypersimplices due to Katzman \cite{katzman-hypersimplices}. As a consequence we provide an elementary proof for the latter (Corollary~\ref{cor:EhrHypersimplex}).

\begin{theorem}\label{thm:main1}
    Let $M$ be a connected matroid of rank $2$. Suppose that $M$ has exactly $s$ hyperplanes of sizes $a_1, \ldots, a_s$. Then $s\geq 3$ and we have
        \[ \ehr(\mathscr{P}(M), t)\ =\ \binom{2t+n-1}{n-1} - \sum_{i=1}^{s} P_{a_i,n}(t)\enspace ,\]
    where
        \[ P_{a,n}(t) := \sum_{k=1}^a \binom{t+n-k-1}{n-k}\binom{t+k-1}{k-1} \, \]
         for $1\leq a\leq n$.
\end{theorem}

A paving matroid is a matroid with the property that all of it subsets of cardinality $\rk(M)-1$ are independent. Our proof relies on the fact that all loopless matroids of rank $2$ on $n$ elements are paving (see Lemma~\ref{lem:seperation}). Theorem~\ref{thm:main1} then allows us to prove Ehrhart positivity of all matroid polytopes of rank $2$. Moreover, we are able to show that the Ehrhart polynomials of connected matroids of rank 2 are coefficient-wise bounded by the matroid polytope of the minimal matroid and the uniform matroid.

For polynomials $p(t), q(t)\in \mathbb{R}[t]$ we write $p(t)\preceq q(t)$ if $q(t)-p(t)$ has only nonnegative coefficients. Let $U_{2,n}$ denote the uniform matroid and $T_{2,n}$ the minimal matroid of rank $2$. With these notations, we prove the following.

\begin{theorem}\label{thm:main}
Let $M$ be a connected matroid of rank $2$ on $n$ elements. Then
    \[
    \ehr(\mathscr{P}(T_{2,n}),t)\ \preceq\ \ehr(\mathscr{P}(M),t)\ \preceq\ \ehr(\mathscr{P}(U_{2,n}),t) \, .
    \]
    Moreover, the inequalities are strict on the coefficients of positive degree whenever the matroid $M$ is neither minimal nor uniform.
    In particular, all matroid base polytopes of rank $2$ matroids are Ehrhart positive.
\end{theorem}

This result proves the aformentioned conjecture of De Loera et al. ~\cite{DeLoera} and the strengthening of that conjecture due to the first author~\cite{ferroni2020ehrhart} for matroids of rank $2$. The key to prove Theorem~\ref{thm:main} is the superadditivity of the polynomials $P_{a_i,n}(t)$ that is provided by Proposition~\ref{prop:superadditivity}.

A further conjecture by De Loera et al.~\cite{DeLoera} concerns the $h^*$-polynomial of matroid polytopes. The $h^*$-polynomial of a lattice polytope is a fundamental tool in Ehrhart theory as it encodes the Ehrhart polynomial in a certain basis with advantageous properties. Given a lattice polytope $\mathscr{P}$ of dimension $d$, the \textbf{$h^*$-polynomial} $h^*(\mathscr{P},x)$ of $\mathscr{P}$ is defined as the numerator polynomial of the generating function
\begin{equation}
    \sum_{j=0}^{\infty} \ehr(\mathscr{P},j)\, x^j\ =\ \frac{h^*(\mathscr{P},x)}{(1-x)^{d+1}}\enspace .
\end{equation}
It can be seen that $h^*(\mathscr{P},x)$ is a polynomial of degree at most $d$ with integer coefficients. 
%The coefficients of the $h^*$-polynomial encode the Ehrhart polynomial in a particular binomial basis. The characterization of $h^*$-polynomials is therefore of central interest in Ehrhart theory.
A foundational result by Stanley~\cite{decompositionsStanley} establishes  that, in contrast to the coefficients of the Ehrhart polynomial, the coefficients of the $h^*$-polynomial are always nonnegative integers. Since then, inequalities amongst the coefficients have been an intensively studied topic ~\cite{PropertyH,PropertyS,Stapledon}. Of particular interest are classes of polytopes for which the $h^*$-polynomial $h^*(\mathscr{P},x)=h_0+h_1x+\cdots +h_dx^d$ has \textbf{unimodal} coefficients, that is
\[
h_0\leq \cdots \leq h_{k-1} \leq h_k \geq h_{k+1} \geq \cdots \geq h_{d}
\]
for some $0\leq k\leq d$. Unimodality and related notions are a fundamental, intensively studied topic not only in Ehrhart theory but more general in
geometric combinatorics as displayed, for example, by groundbreaking work by Adiprasito, Huh and Katz~\cite{AdiprasitoHuhKatz} proving the Heron-Rota-Welsh conjecture and Br\"anden and Huh~\cite{BrandenHuh} on Lorentzian polynomials leading to a proof of the strongest of Mason's conjectures. For further reading see also the surveys~\cite{Braun,BrandenUnimodality,BrentiUnimodality,StanleyUnimodality}.

In~\cite{DeLoera} De Loera et al. provide a closed formula for the $h^*$-polynomial of hypersimplices, which are precisely the matroid polytopes of uniform matroids. From this formula they derive that the $h^*$-polynomial of the second hypersimplex has unimodal coefficients. Based on further computational experiments they pose the intriguing conjecture that all matroid polytopes have $h^*$-polynomials with unimodal coefficients. One property of polynomials that implies unimodality of its coefficients is \textbf{real-rootedness}, that is, if the polynomial possesses only real zeros. Even more, if a polynomial $h_0+h_1x+\cdots +h_dx^d$ has only real roots then its coefficients form a \textbf{log-concave} sequence, that is, $h_i^2\geq h_{i-1}h_{i+1}$ for all $0 < i < d$. Hence, one way of showing that a sequence is unimodal or log-concave is to show that it is the sequence of  coefficients of a real-rooted polynomial. See, e.g., \cite{Braenden:2015} for more on real-rootedness and related properties. In
~\cite{ferroni2020ehrhart} it was shown that matroid polytopes of minimal matroids have real-rooted $h^*$-polynomials and, supported by further experiments, it was conjectured that all matroid polytopes have real-rooted $h^*$-polynomials. Towards this conjecture we prove that all matroid polytopes of all sparse paving matroids of rank $2$, which include the second hypersimplex, have real-rooted $h^*$-polynomials.
\begin{theorem}\label{thm:main3} 
    Let $M$ be a sparse paving matroid of rank $2$. Then $h^*(\mathscr{P}(M),x)$ is either $1$, $1+x$, or a real-rooted polynomial of degree $\big\lfloor\frac{n}{2}\big\rfloor$ with positive coefficients.
\end{theorem}
In particular, it follows that the sequence of coefficients of the $h^*$-polynomials for sparse paving matroids is unimodal and log-concave. Our proof is based on the closed formula for the $h^*$-polynomial for all rank $2$ matroids that we provide in Proposition~\ref{main_but_for_h_star} in conjunction with a careful analytic discussion of this formula in case of sparse paving matroids.\\

\noindent \textbf{Outline:} We begin by collecting preliminaries on matroids, their base polytopes and Ehrhart theory in Section~\ref{sec:prelim}. Section~\ref{sec:EhrhartPoly} is dedicated to the proof of Theorem~\ref{thm:main1}. Section~\ref{sec:Positivity} is concerned with the proof of Theorem~\ref{thm:main}. In Section~\ref{sec:h_star} we adress the proof of Theorem~\ref{thm:main3}.
We close this article with open questions in Section~\ref{sec:Outlook}. 

\section{Preliminaries}\label{sec:prelim}
\noindent In this section we collect basic preliminaries on Ehrhart theory, matroids and their polytopes. We restrict ourselves to introduce only the most relevant terms and the polyhedral point of view on matroids. The focus in this article lies on the special case of rank $2$ matroids. 

\subsection{Matroids} Matroids have been developed by Whitney
\cite{Whitney:1935} in 1935 and independently by Nakasawa, see \cite{Nakasawa:2009}.
They generalize the concept of independence and dependence in graphs, linear vector spaces and algebraic extensions. For further reading on matroids we recommend the monographs by Oxley \cite{Oxley} and White \cite{White:1986}. For a polyhedral point of view on the topic we recommend the book \cite{Schrijver:2003:B} by Schrijver.

There are many equivalent ways to define a matroid, see the appendix of \cite{White:1986} for an overview of those cryptomorphisms. Let $k$ be a nonnegative integer and $E$ be a finite set. A nonempty collection $\mathscr{B}$ of $k$-subsets of $E$ defines a \textbf{matroid} $M$ of \textbf{rank} $k$, denoted by $\rk(M)$, on the \textbf{ground set} $E$ with set of  \textbf{bases} $\mathscr{B}$ whenever the following exchange property is satisfied:

\begin{center}
    For all $B_1,B_2\in\mathscr{B}$ and $i\in B_1$ exists $j\in B_2$ such that $(B_1\setminus \{i\} )\cup \{j\}\in\mathscr{B}$. 
\end{center}

The \textbf{rank} of a set $S\subseteq E$ in $M$, denoted $\rk(S)$, is the maximal size of a set $S\cap B$ where $B$ ranges over the family $\mathscr{B}$ of all bases of $M$. A \textbf{flat} of $M$ is a subset $F\subseteq E$ such that for each element $e\in E\setminus F$ the rank of $F\cup\{e\}$ is strictly larger than the rank of~$F$. The flats of rank $\rk(M) -1$ are called \textbf{(matroid) hyperplanes}. A subset $I\subseteq E$ is an \textbf{independent set} if it is contained in some $B\in\mathscr{B}$, otherwise it is a \textbf{dependent set}. A \textbf{loop} is an element of $E$ that is contained in no basis. If $e$ and $f$ are not loops, we say that they are \textbf{parallel} if the set $\{e,f\}$ is dependent.
We call a matroid of rank $k$ \textbf{paving} if all of its subsets of size $k-1$ are independent. In particular, rank $2$ matroids without loops are paving. A matroid is called \textbf{sparse paving} if additionally every matroid hyperplane is at most of size $k$.

Let $M_1$ be a matroid with ground set $E_1$ and set of bases $\mathscr{B}_1$ and let $M_2$ be a matroid on $E_2$ with bases $\mathscr{B}_2$. If the sets $E_1$ and $E_2$ are disjoint then the collection
\[
    \mathscr{B}:=\{B_1\sqcup B_2 : B_1\in \mathscr{B}_1 \text{ and } B_2\in \mathscr{B}_2\}
\]
is the family of bases of the matroid $M_1\oplus M_2$ which is called the \textbf{direct sum} of $M_1$ and $M_2$. A matroid is \textbf{disconnected} if it is a direct sum of matroids with smaller cardinality and \textbf{connected} otherwise. The rank of a direct sum is the sum of the ranks of the summands.

\begin{example}
    The maximum number of bases that a rank $k$ matroid on $n$ elements can have is $\binom{n}{k}$. This bound is achieved whenever each $k$-set is a basis. The corresponding matroid is called the \textbf{uniform matroid}, denoted $U_{k,n}$.
\end{example}

\begin{example}\label{ex:minimalMatroid}
    The minimum number of bases of a connected matroid on $n$ elements of rank $k$ is $k\cdot (n-k)+1$. Being connected requires that $n=1$ whenever $k=0$ or $n=k$. The \textbf{minimal matroid} $T_{k,n}$ is the unique matroid up to isomorphisms achieving this minimum, see \cite{Dinolt:1971} or \cite{Murty:1971}. Let $S=\{1,\ldots,k\}$; the collection formed by the sets $(S\setminus\{i\})\cup\{j\}$ where $i\in S$ and $j\in E\setminus S$ together with the set $S$ is the collection of bases of $T_{k,n}$. In particular, the elements $k+1,\ldots,n$ are parallel and form a flat of rank one.
\end{example}

\subsection{Matroid polytopes}
\noindent The convex hull of all indicator vectors of the bases of $M$ form the \textbf{matroid (base) polytope}:
    \[ \mathscr{P}(M) := \operatorname{conv}\{e_B : B\in \mathscr{B}\}\enspace \]
where $e_B := \sum_{i\in B} e_i$ is the indicator vector of the basis $B\in \mathscr{B}$.
For any matroid $M$ on $n$ elements, the matroid polytope $\mathscr{P}(M)$ has the following outer description:
\begin{equation}\label{ineq:flats}
\mathscr{P}(M)\ =\ \left\{ x\in [0,1]^n : \sum_{i\in F} x_i \leq \rk(F) \text{ for all flats $F$ of } M \text{ and } \sum _{i=1}^n x_i=\rk (M) \right\} .
\end{equation}

Notice that the polytope $\mathscr{P}(M)$ is of dimension at most $n-1$, as it lies on the hyperplane $\sum_{i=1}^n x_i = \rk(M)$. Furthermore,
the dimension of the polytope $\mathscr{P}(M)$ equals to $n-1$ if and only if the matroid is connected; see \cite{Fujishige:1984} or \cite{feichtner-sturmfels}.

\begin{example}
The matroid polytope of the uniform matroid $U_{k,n}$ is the \textbf{$k$-th hypersimplex}
\[
\Delta_{k,n} := \mathscr{P}(U_{k,n}) = 
\left\{ x\in[0,1]^n : \sum_{i=1}^n x_i = k \right\} \enspace .
\]
This is a point if $k=0$ or $k=n$ and a unimodular simplex whenever $k=1$ or $k=n-1$.
\end{example}

\begin{example} 
The matroid polytope of the minimal matroid $T_{k,n}$ is given by 
\[
\mathscr{P}(T_{k,n})\ =\ \left\{ x\in\Delta_{k,n} : \sum_{i=k+1}^n x_i \leq 1 \right\} \enspace ,
\]
see \cite[Proposition 2.6]{ferroni2020ehrhart}.
\end{example}

If the matroid $M$ is a direct sum $M_1\oplus M_2$ then its polytope $\mathscr{P}(M)$ equals the product $\mathscr{P}(M_1)\times \mathscr{P}(M_2)$ of the matroid polytopes $\mathscr{P}(M_1)$ and $\mathscr{P}(M_2)$.

We will now focus on the case of rank $2$ matroids.
\begin{example}
The matroid polytopes $\mathscr{P}(U_{0,1})=\Delta_{0,1}$ and $\mathscr{P}(U_{1,1})=\Delta_{1,1}$ are points. 
Thus the matroid polytope of the directed sum $M\oplus U_{0,1}$ or $M\oplus U_{1,1}$ is a unimodular equivalent embedding of the polytope $\mathscr{P}(M)$ in a higher dimensional space.
\end{example}

Note that if a matroid $M$ has a loop, then $M$ is a direct sum $M=M'\oplus U_{0,1}$. 
As loops do not change the matroid polytope, only their embedding, we may assume from now on that all matroids that we consider are loopless. We benefit of the following fact.

\begin{lemma}\label{lem:disconnected} Let $M$ be a matroid of rank $2$ with no loops. Then $M$ is either connected or a direct sum of two uniform matroids of rank one. In particular, the matroid polytope of the latter is a product of two simplices. 
\end{lemma}

The flats of a rank $2$ matroid $M$ are the set of all loops, the hyperplanes and the ground set.
If $M$ is loopless or connected, then the set of loops is empty. Neither the empty set nor the ground set impose a facet defining inequality in the description~\eqref{ineq:flats}. Thus we obtain the following simplification of \eqref{ineq:flats} for a loopless matroid of rank 2 on a ground set of size $n$.
\begin{equation}\label{ineq:hyperplanes}
\mathscr{P}(M)\ =\ \left\{ x\in\Delta_{2,n} : \sum_{i\in H} x_i \leq 1 \text{ for all matroid hyperplanes $H$ of } M \right\} \enspace .
\end{equation}

Being paving is a key property of connected rank $2$ matroids. Under such hypothesis on the rank, being paving is equivalent to being loopless. Also, notice that for loopless matroids the set of hyperplanes provides a partition of the ground set. This fact will be used several times in the sequel. One way of capturing such property geometrically is the following Lemma.

\begin{lemma}\label{lem:seperation} Let $M$ be a loopless matroid of rank $2$ and $u\in \Delta_{2,n}\setminus\mathscr{P}(M)$. Then $u$ violates exactly one of the inequalities
\[
\sum_{i\in H} x_i \leq 1
\]
where $H$ is a matroid hyperplane of $M$.
\end{lemma}
\begin{proof} Clearly $u\in\Delta_{2,n}\setminus\mathscr{P}(M)$ has to violate at least one of the above inequalities. Suppose $u$ satisfies
\[
\sum_{i\in H} u_i > 1 \quad \text{ and } \quad \sum_{i\in G} u_i > 1
\]
where $G$ and $H$ are distinct matroid hyperplanes. The intersection $G\cap H$ is empty as $M$ has no loops. Therefore 
\[
2\ <\ \sum_{i\in H} u_i + \sum_{i\in G} u_i\ \leq\ \sum_{i=1}^n u_i \enspace .
\]
Contradicting that the coordinate sum of $u$ is $2$ whenever $u\in\Delta_{2,n}$.
\end{proof}

In \cite{JoswigSchroeter:2017} Joswig and the third author introduce the class of split matroids which provides the same separation property in arbitrary rank. 
This class strictly contains all paving matroids and thus include the loopless matroids of rank $2$. 
Moreover, that article contains further details about matroid polytopes and their facets.

\subsection{Ehrhart theory}
\noindent In 1962 Ehrhart \cite{Ehrhart} initiated the study of lattice-point enumeration in dilations of lattice polytopes with the following foundational result.

\begin{theorem}[Ehrhart's Theorem]\label{thm:Ehrhart}
    Let $\mathscr{P}\subseteq \mathbb{R}^n$ be a lattice polytope of dimension $d$. There is a polynomial $\ehr(\mathscr{P},t)$ in the variable $t$ of degree $d$ such that the number of lattice points in the $t$-th dilate $t\,\mathscr{P}=\{t\, p\, \colon p\in \mathscr{P}\}$ satisfies
    \[
    \ehr(\mathscr{P}, t)\ =\ \#(t\mathscr{P}\cap\mathbb{Z}^n)
    \]
    for all integers $t\geq 0$.
\end{theorem}

The polynomial $\ehr(\mathscr{P},t)$ is called the \textbf{Ehrhart polynomial} of $\mathscr{P}$. For a proof of Ehrhart's theorem and further reading on integer point enumeration we refer to \cite{beck-robins}.

A subset $\widetilde{\mathscr{P}}\subseteq\mathbb{R}^n$ obtained from a convex polytope $\mathscr{P}\subseteq \mathbb{R}^n$ by removing some of its facets is called a \textbf{half-open polytope}. We note that Ehrhart's Theorem~\ref{thm:Ehrhart} naturally extends to half-open lattice polytopes via the inclusion-exclusion principle. That is, the number of lattice points in positive integer dilations of a half-open lattice polytope $\widetilde{\mathscr{P}}$ is also given by a polynomial in the variable $t$.

If the Ehrhart polynomial $\ehr(\mathscr{P},t)$ has positive coefficients, we say that the polytope~$\mathscr{P}$ is \textbf{Ehrhart positive}. Moreover, we define the following partial order $\preceq$ on the ring of polynomials $\RR[t]$. The polynomial $p(t) = \sum_{j=0}^d a_j t^j$ is said to be \textbf{nonnegative} if all its coefficients are nonnegative, that is, $a_j\geq 0$ for all $j\geq 0$. In this case, we write $p(t) \succeq 0$. Furthermore, we write $p(t)\succeq q(t)$ whenever $p(t)-q(t)\succeq 0$. We say that the inequality is \textbf{strict on the coefficients of positive degree} if $p(t)-q(t)$ has only positive coefficients, except for possibly the constant coefficient which might be zero.

We observe that $\preceq$ defines a partial order that is preserved under multiplication with nonnegative polynomials. That is, for $p,q,r\in\RR[t]$ and $r(t)\succeq 0$ 
    \begin{equation}\label{rem:multiplication}
         p(t)\preceq q(t) \implies p(t)\cdot r(t) \preceq q(t)\cdot r(t) \enspace. 
    \end{equation}

\begin{example} The $t$-th dilation of the simplex $\Delta_{1,n}=\mathscr{P}(U_{1,n})$ is
\[
t\,\Delta_{1,n}\ =\ \left\{x\in[0,t]^n : \sum_{i=1}^n x_i = t\right\}
\]
which contains $\binom{t+n-1}{n-1}$ lattice points.
Hence, the Ehrhart polynomial of a loopless matroid of rank $1$ is equal to
\[
\ehr(\mathscr{P}(U_{1,n}),t)\ =\ \binom{t+n-1}{n-1}\ =\ \prod _{i=0}^{n-2}\frac{t+n-1-i}{n-1-i}
\]
which shows that $\mathscr{P}(U_{1,n})$ is Ehrhart positive.

The half-open simplex 
\[
\widetilde{\Delta}_{1,n}\ =\ \left\{x\in[0,1]^n : \sum_{i=1}^n x_i = 1 \text{ and } x_1>0 \right\}
\]
is equal to the set difference $\mathscr{P}(U_{1,n})\setminus \mathscr{P}(U_{1,n-1}\oplus U_{0,1})$. It follows that
\[
\ehr(\widetilde{\Delta}_{1,n})\ = \ \binom{t+n-1}{n-1}-\binom{t+n-2}{n-2}\ =\ \binom{t+n-2}{n-1}\ = \ \prod _{i=0}^{n-2}\frac{t+n-2-i}{n-1-i}\ \succeq\ 0
\]
as $\ehr(\mathscr{P}(U_{1,n-1}\oplus U_{0,1}),t) = \ehr(\mathscr{P}(U_{1,n-1}),t)$.
In particular, we have $\ehr(\mathscr{P}(U_{1,n}),t) \succeq \ehr(\mathscr{P}(U_{1,n-1}),t)$.
\end{example}

\begin{example}\label{ex:ehrDisconnected}
    Lemma~\ref{lem:disconnected} implies that the matroid polytope of a disconnected matroid of rank $2$ without loops is a product of two simplices $\mathscr{P}(U_{1,m}\oplus U_{1,n-m})=\Delta_{1,m}\times\Delta_{1,n-m}$ for some $1\leq m\leq n-1$.
    Its Ehrhart polynomial is therefore given by 
    \[ \ehr(\mathscr{P}(U_{1,m}\oplus U_{1,n-m}),t)\ =\ \binom{t+m-1}{m-1}\binom{t+n-m-1}{n-m-1}\] which is positive, as it is a product of linear factors with positive coefficients.
\end{example}

\begin{rem}\label{rem:EhrPositivity} 
    Example~\ref{ex:ehrDisconnected} shows that base polytopes of disconnected matroids of rank~$2$ without loops are Ehrhart positive. On the other hand, extending a matroid $M$ with $m$ loops, that is, considering the matroid $M\oplus U_{0,m}$, does not change the Ehrhart polynomial. In order to prove Ehrhart positivity of rank $2$ matroids it therefore suffices to consider connected matroids only.
\end{rem}

An important tool to study Ehrhart polynomials further is the $h^*$-polynomial of a lattice polytope. Given a lattice polytope $\mathscr{P}$ of dimension $d$, the $h^*$-polynomial $h^* (\mathscr{P}, x)$ is defined as the numerator of the generating function
\begin{equation}\label{h_star_def}
    \sum_{j=0}^{\infty} \ehr(\mathscr{P},j)\, x^j\ =\ \frac{h^*(\mathscr{P},x)}{(1-x)^{d+1}}\enspace .
\end{equation}
General theory for generating functions implies that $h^* (\mathscr{P}, x)$ is a polynomial of degree at most $d$ (see, e.g.,~\cite[Corollary 4.3.1]{EC1}). Since the evaluations $\ehr(\mathscr{P},j)$ are integers so are the coefficients of the numerator polynomial. Equivalently, $h^* (\mathscr{P}, x)=h_0+h_1x+\cdots +h_dx^d$ is the $h^*$-polynomial of $\mathscr{P}$ if and only if 
        \[ \ehr(\mathscr{P},t)\ =\ h_0\binom{t+d}{d} + h_1\binom{t+d-1}{d} + \ldots + h_{d-1}\binom{t+1}{d} + h_d\binom{t}{d} \enspace .\]
(See, e.g.,~\cite{beck-robins}). That is, the coefficients $h_0,h_1,\ldots, h_d$ of the $h^*$-polynomial encode the Ehrhart polynomial with respect to the basis $\binom{t+d}{d}, \binom{t+d-1}{d}, \ldots \binom{t}{d}$ of the vector space of real polynomials of degree at most~$d$. By a fundamental theorem of Stanley~\cite{decompositionsStanley}, these coefficients are always nonnegative, in contrast to the coefficients of the Ehrhart polynomial.

The following well-known binomial identity will serve us as a useful tool to compute $h^*$-polynomials of matroid polytopes. (For a proof, see, e.g., \cite[Corollary 2]{szekely}.)
\begin{lemma}[Sur\'anyi's Identity]\label{lem:Suranyi} For all natural numbers $r$ and $s$
    \[
        \binom{t+r}{r}\binom{t+s}{s} \ =\ \sum_{j \geq 0} \binom{r}{j} \binom{s}{j} \binom{t+r+s-j}{r+s}
    \]
\end{lemma}

\begin{example}\label{ex:hstarDisconnected} Lemma~\ref{lem:Suranyi} provides us with an alternative formula of the Ehrhart polynomial of two simplices given in  Example~\ref{ex:ehrDisconnected} above. For $1\leq m \leq n-1$ a direct application of Sur\'anyi's Identity with $r=m-1$ and $s=n-m-1$ implies
\[
\ehr(\mathscr{P}(U_{1,m}\oplus U_{1,n-m}),t)\ =\ \sum_{j \geq 0} \binom{m-1}{j} \binom{n-m-1}{j} \binom{t+n-2-j}{n-2} \, .
\]
Equivalently,
\[
h^*(\mathscr{P}(U_{1,m}\oplus U_{1,n-m}),x)\ = \ \sum_{j \geq 0} \binom{m-1}{j} \binom{n-m-1}{j} \, x^j \enspace .
\]
\end{example}

\section{Ehrhart polynomials}\label{sec:EhrhartPoly}
\noindent In this section we give a proof of Theorem~\ref{thm:main1}. 

We consider the polytopes
\[ \mathscr{Q}_{k,n} = \left\{x\in \Delta_{2,n}: \sum_{i=1}^{k-1} x_i \leq 1, \sum_{i=k+1}^n x_i \leq 1\right\}\enspace \]
for all $1\leq k\leq n-1$, together with their half-open version
\[ \widetilde{\mathscr{Q}}_{k,n} := \left\{x\in \Delta_{2,n}: \sum_{i=1}^{k-1} x_i \leq 1, \sum_{i=k+1}^n x_i < 1\right\}\enspace .\]
Observe that for $k=1$, $\mathscr{Q}_{1,n}$ is isomorphic to $\Delta_{1,n-1}$ and  $\widetilde{\mathscr{Q}}_{1,n}$ is the empty polytope.

\begin{remark}\label{rem:graphic} 
    The polytope $\mathscr{Q}_{k,n}$ is the matroid polytope of a rank $2$ matroid on $n$ elements where the first $k-1$ elements are parallel and the last $n-k$ elements are parallel. This particular matroid is induced by a graph. This graph consists of a cycle of length three whenever $k>1$ to which several parallel edges have been added as follows, there is one copy of one edge, $n-k$ parallel copies of another edge, and $k-1$ parallel copies of a third edge.

    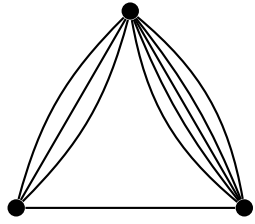
\begin{figure}[t] 
		\begin{tikzpicture} 
		[scale=3.0,auto=center,every node/.style={circle, fill=black, inner sep=2.3pt}] 
		\tikzstyle{edges} = [black, thick];
		
		\node (a2) at (-0.5,0.86) {};  
		\node (a3) at (0,1.73)  {}; 
		\node (a1) at (0.5,0.86)  {};    
		
		\draw[edges] (a2) -- (a1);
		\draw[edges] (a3) -- (a1);
		\draw[edges] (a3) edge[bend right=20] (a1) ;
		\draw[edges] (a3) edge[bend right=-20] (a1);
		\draw[edges] (a3) edge[bend right=8] (a1) ;
		\draw[edges] (a3) edge[bend right=-8] (a1);
		\draw[edges] (a3) edge[bend right=15] (a2);
		\draw[edges] (a3) edge[bend right=-15] (a2);
		\draw[edges] (a3) -- (a2);
		\end{tikzpicture}
		\caption{The graph of Remark~\ref{rem:graphic} with $n = 9$ edges and $k=4$.}\label{fig:graph}
	\end{figure}
    
    Figure \ref{fig:graph} depicts this graph for the case $n=9$ and $k=4$. These matroids fall into the well studied class of lattice path matroids. More precisely they are the snakes $\mathsf{S}(k-1, 2, n-k-1)$ in the notation of \cite{knauer18}. Notice that the snake $\mathsf{S}(1,2, n-3)$ is the minimal matroid $T_{2,n}$ of Example~\ref{ex:minimalMatroid}.
\end{remark}
    
    We obtain the following formulas for the Ehrhart polynomials of the matroid polytope~$\mathscr{Q}_{k,n}$ and the half-open polytope~$\widetilde{\mathscr{Q}}_{k,n}$.

\goodbreak

\begin{proposition}\label{prop:ehrQ} 
    For all $1\leq k\leq n-1$
    \begin{eqnarray*}
        \ehr(\mathscr{Q}_{k,n},t) &=& \binom{t+k-1}{k-1}\binom{t+n-k}{n-k} - \binom{t+n-2}{n-1} \, ,\quad \text{ and }\\
        \ehr(\widetilde{\mathscr{Q}}_{k,n},t) &=& \binom{t+k-1}{k-1} \binom{t+n-k-1}{n-k} - \binom{t+n-2}{n-1} \enspace .
    \end{eqnarray*}
\end{proposition}

\begin{proof}
    By definition we have
    \begin{align*}
        \ehr(\mathscr{Q}_{k,n}, t)\ &=\ \# (t\mathscr{Q}_{k,n}\cap \mathbb{Z}^n)\\
        &=\ \# \left\{x\in [0,t]^n \cap \mathbb{Z}^n : \sum_{i=1}^{n} x_i = 2t, \sum_{i=1}^{k-1} x_i \leq t, \sum_{i=k+1}^n x_i\leq t\right\} \enspace . 
    \end{align*}
    The above expression can be interpreted as the number of ways of placing $2t$ indistinguishable balls into $n$ distinct boxes, each of capacity $t$, under the additional constraints that the first $k-1$ as well as the last $n-k$ boxes together contain at most $t$ balls. The number $x_i$ equals the number of balls in box $i$ in this setting. 
    
    As a first step, we ignore the capacity bound $x_k\leq t$ for a moment, and count the number of ways that $t$ balls can be placed into the first $k$ boxes, and the remaining $t$ balls are placed into the last $n-k+1$ boxes. There are $\binom{t+k-1}{k-1}\binom{t+n-k}{n-k}$ ways of placing $2t$ balls in such a way. (Notice that this number does not count a distribution more than once, as the number of balls placed in box $k$ in the first batch can be recovered from the number of balls in the boxes $1$ to $k-1$, and similarly for the second batch.)
    
    As a second step we count in how many cases we placed more than $t$ balls in box $k$. In these cases the $k$-th box contains at least $t+1$ many balls. If we place $t+1$ many balls in box $k$, there are $\binom{t+n-2}{n-1}$ many possibilities to place the remaining $2t - (t+1) = t-1$ balls into all $n$ boxes. Subtracting this number from the above leads to the first formula.
    
    To obtain the second formula we observe that the polytope $\mathscr{Q}_{k,n}$ is the disjoint union of~$\widetilde{\mathscr{Q}}_{k,n}$ and the product of simplices
    \[
    \left\{x\in [0,1]^n: \sum_{i=1}^{k} x_i = 1, \sum_{i=k+1}^n x_i = 1\right\} \ =\ \Delta _{1,k} \times \Delta _{1,n-k}
    \]
    whose Ehrhart polynomial is equal to $\binom{t+k-1}{k-1} \binom{t+n-k-1}{n-k-1}$.
    It follows that 
    \begin{align*}
       \ehr(\widetilde{\mathscr{Q}}_{k,n},t)\ &=\ \ehr(\mathscr{Q}_{k,n},t) - \binom{t+k-1}{k-1} \binom{t+n-k-1}{n-k-1}\\
       &=\ \binom{t+k-1}{k-1}\left( \binom{t+n-k}{n-k}-\binom{t+n-k-1}{n-k-1} \right) - \binom{t+n-2}{n-1}\\
        &=\ \binom{t+k-1}{k-1} \binom{t+n-k-1}{n-k} - \binom{t+n-2}{n-1}
    \end{align*}
    as desired.
\end{proof}

We observe that the polytope $\mathscr{Q}_{2,n}$ agrees with the matroid polytope of the minimal matroid $T_{2,n}$. From Proposition~\ref{prop:ehrQ} we therefore obtain an alternative proof for the Ehrhart polynomial of the minimal matroid $T_{2,n}$ given in~\cite[Theorem 3.1]{ferroni2020ehrhart}.

\begin{corollary}\label{cor:ehrMinimal}
The Ehrhart polynomial of matroid polytope of the minimal matroid $T_{2,n}$ equals
\[
\ehr(\mathscr{P}(T_{2,n}),t)\ =\ \binom{t+n-1}{n-1} + (n-3)\binom{t+n-2}{n-1} \enspace .
\]
\end{corollary}
\begin{proof} By Proposition~\ref{prop:ehrQ}, we have
\begin{align*}
    \ehr(\mathscr{P}(T_{2,n}),t)\ &=\ \ehr(\mathscr{Q}_{2,n},t) \ =\  (t+1) \binom{t+n-2}{n-2} - \binom{t+n-2}{n-1}\\
    &=\ t \binom{t+n-2}{t} + \binom{t+n-2}{n-2} - \binom{t+n-2}{n-1}\\
    %&= (t+n-2)\binom{t+n-3}{n-2} + \binom{t+n-1}{n-1} - 2\cdot \binom{t+n-2}{n-1}\\
    &=\ (n-1)\binom{t+n-2}{n-1} + \binom{t+n-1}{n-1} - 2\, \binom{t+n-2}{n-1}\\
    &=\ \binom{t+n-1}{n-1} + (n-3)\binom{t+n-2}{n-1} \, 
\end{align*}
which proves the claim.
\end{proof}

For $1\leq \ell \leq n-1$ we now consider the half-open polytope

\[ \widetilde{\mathscr{R}}_{\ell,n} :=  \left\{ x\in \Delta_{2,n} : \sum_{i=1}^\ell x_i > 1\right\} = \left\{ x\in \Delta_{2,n} : \sum_{i=\ell+1}^n x_i < 1\right\} \enspace .\]
Observe that $\widetilde{\mathscr{R}}_{1,n}$ agrees with $\widetilde{\mathscr{Q}}_{1,n}$ which is the empty polytope. Furthermore, note that each of the polytopes $\widetilde{\mathscr{R}}_{\ell,n}$ can be decomposed as \[
\widetilde{\mathscr{R}}_{\ell,n}\ =\ \widetilde{\mathscr{Q}}_{1,n}\sqcup \widetilde{\mathscr{Q}}_{2,n}\sqcup \cdots \sqcup \widetilde{\mathscr{Q}}_{\ell,n} \enspace .\]

For all $n\geq 0$ and $1\leq a \leq n$ we define the polynomials
\[P_{a,n}(t):= \sum_{k=1}^a \binom{t+n-k-1}{n-k}\binom{t+k-1}{k-1}\enspace .\]
Furthermore, we set $P_{0,n}(t):=0$ for all $n\geq 0$. 

As a direct consequence of Proposition~\ref{prop:ehrQ} we obtain the following.

\begin{corollary}\label{cor:EhrR}
    For all $1\leq \ell \leq n-1$ the Ehrhart polynomial of $\widetilde{\mathscr{R}}_{\ell,n}$ equals
    \[ \ehr(\widetilde{\mathscr{R}}_{\ell,n},t) = P_{\ell, n}(t)-\ell \binom{t+n-2}{n-1}.\]
\end{corollary}

Similarly, we may decompose the second hypersimplex $\Delta_{2,n}$ as
\begin{equation}\label{eq:decomposition-hypersimplex}
    \Delta_{2,n}  = \widetilde{\mathscr{R}}_{n-1,n}\sqcup \Delta _{1,n-1} = \widetilde{\mathscr{Q}}_{1,n}\sqcup \widetilde{\mathscr{Q}}_{2,n}\sqcup \cdots \sqcup \widetilde{\mathscr{Q}}_{n-1,n} \sqcup \Delta _{1,n-1} \enspace .
\end{equation}
This decomposition allows us to give a simple proof for the known formula for the Ehrhart polynomial of second hypersimplices due to Katzman \cite{katzman-hypersimplices}.

\begin{corollary}\label{cor:EhrHypersimplex}
The Ehrhart polynomial of the hypersimplex $\Delta_{2,n}$ is given by
        \[ \ehr(\Delta_{2,n},t) = \binom{2t+n-1}{n-1} - n \binom{t+n-2}{n-1}.\]
\end{corollary}

\begin{proof}
From Equation~\eqref{eq:decomposition-hypersimplex} and Proposition~\ref{prop:ehrQ} we obtain
\begin{eqnarray*}
        \ehr(\Delta_{2,n},t)
        &=& \sum_{k=1}^{n-1} \binom{t+n-k-1}{n-k}\binom{t+k-1}{k-1} - (n-1)\binom{t+n-2}{n-1}+\binom{t+n-2}{n-2}\\
        &=& \sum_{k=1}^{n} \binom{t+n-k-1}{t-1} \binom{t+k-1}{t}- n\binom{t+n-2}{n-1}\nonumber\\
        &=& \binom{2t+n-1}{n-1}-n\binom{t+n-2}{n-1} \, .
\end{eqnarray*}
where in the last step we used a variation of the Chu–Vandermonde identity on binomial coefficients which, for example, can be found in~\cite[(5.26) on page 169]{graham-knuth-patashnik}.
\end{proof}

We are now prepared to prove Theorem~\ref{thm:main1}.
\begin{proof}[Proof of Theorem~\ref{thm:main1}] First note that a rank $2$ matroid is disconnected whenever it has only $s\leq 2$ hyperplanes.
Moreover, the ground set of a connected matroid $M$ of rank~$2$ with $s$ hyperplanes has at least $s\geq 3$ elements, and a connected matroid on $n\geq 2$ elements is loopless. Thus formula~\eqref{ineq:hyperplanes} applies and hence the matroid polytope of $M$ is
    \[ \mathscr{P}(M)\ =\ \left\{ x\in \Delta_{2,n} : \sum_{i\in H} x_i \leq 1 \text{ for every $H$ hyperplane}\right\}\enspace .\]
    Furthermore, the matroid hyperplanes of a loopless rank $2$ matroid partition the ground set.
    Now pick any hyperplane $H$ of cardinality $a_r$.
    The subset of $\Delta_{2,n}$ that violates the inequality for $H$ is a copy of $\widetilde{\mathscr{R}}_{a_r,n}$ after permuting the coordinates. 
    Moreover, Lemma~\ref{lem:seperation} shows that a point in $\Delta_{2,n}$ can violate at most one inequality imposed by a hyperplane.
    
     Therefore, by applying the formulas for the Ehrhart polynomials of Corollary~\ref{cor:EhrR} and \ref{cor:EhrHypersimplex} we obtain:
    \begin{align*}
        \ehr(\mathscr{P}(M),t)\ &=\ \ehr(\Delta_{2,n},t) - \sum_{i=1}^{s} \ehr(\widetilde{\mathscr{R}}_{a_i,n}, t)\\
        &=\ \left(\binom{2t+n-1}{n-1} - n \binom{t+n-2}{n-1}\right) - \sum_{i=1}^s \left(P_{a_i,n}(t) - a_i\binom{t+n-2}{n-1}\right)\\
        &=\ \binom{2t+n-1}{n-1} - \sum_{i=1}^s P_{a_i,n}(t)
    \end{align*}
    where in the last step we used $a_1+\cdots + a_s=n$ which is satisfied since the hyperplanes form a partition of the ground set.
\end{proof}

\section{Ehrhart positivity}\label{sec:Positivity}
\noindent The purpose of this section is to prove Theorem~\ref{thm:main}. Our proof rests on the following superadditivity of the polynomials $P_{a,n}$.

\begin{proposition}\label{prop:superadditivity}
For all nonnegative integers $a, b, n$ such that $a+b\leq n$
    \[P_{a,n} + P_{b,n}\ \preceq\ P_{a+b,n} \enspace .\]
Moreover, the inequality on the coefficients of positive degree is strict whenever $a,b>0$.
\end{proposition}

\begin{proof} There is nothing to show if $a=0$. Thus fix numbers $1\leq a\leq b$ such that $a+b\leq n$. We are going to prove that 
        \begin{equation}\label{iteration} P_{a,n} + P_{b,n}\ \preceq\ P_{a-1,n} + P_{b+1,n}\enspace .\end{equation}
    This will prove the claim since applying this inequality $a$ times yields
        \[ P_{a,n}+P_{b,n}\ \preceq\ P_{a-1,n} + P_{b+1,n}\ \preceq\ P_{a-2,n} + P_{b+2,n}\ \preceq\ \cdots\ \preceq\ P_{0,n} + P_{a+b,n}\ =\ P_{a+b,n}\enspace .\]
    Moreover, our proof will show that in \eqref{iteration} the inequality on the coefficients of positive degree is strict. 
    Inequality \eqref{iteration} is equivalent to
        \[ P_{a,n} - P_{a-1,n}\ \preceq\ P_{b+1,n} - P_{b,n}\enspace ,\]
    which, by definition, is equivalent to
        \begin{equation}\label{eq:inequality1}
            \binom{t+n-a-1}{n-a}\binom{t+a-1}{a-1}\ \preceq\ \binom{t+n-b-2}{n-b-1}\binom{t+b}{b}\enspace .
        \end{equation}
    Notice that both sides have the common factor $\binom{t+n-b-2}{n-b-1}\binom{t+a-1}{a-1}$ which has nonnegative coefficients.
    After canceling this factor and multiplication with the positive number $\binom{b}{b-a+1}\binom{n-a}{b-a+1}$, we obtain the inequality
         \begin{equation}\label{eq:inequality2}
            \binom{t+n-a-1}{b-a+1} \binom{b}{b-a+1}\ \preceq\ \binom{t+b}{b-a+1}\binom{n-a}{b-a+1}\enspace .
         \end{equation}
    Inequality~\eqref{eq:inequality1} is implied by \eqref{eq:inequality2} using property \eqref{rem:multiplication}. Also, notice that if we prove that \eqref{eq:inequality2} is strict for all coefficients, then \eqref{eq:inequality1} is strict for all coefficients of positive degree. This is because the polynomial $\binom{t+n-b-2}{n-b-1}\binom{t+a-1}{a-1}$ is a product of $t$ and a polynomial with positive coefficients.
    
    To prove this, we use the following variables $c = n-a$ and $u=b-a+1$.
    Since $a+b\leq n$ we have $b\leq c$. Moreover, we have $1\leq u \leq b$. Observe that inequality~\eqref{eq:inequality2} reads
    \begin{equation}
        \binom{t+c-1}{u} \binom{b}{u}\ \preceq\ \binom{t+b}{u}\binom{c}{u}\enspace ,
    \end{equation}
    after substitution.  Observe further that if $b=c$, then the inequality is automatically satisfied, and is in fact strict on all coefficients. Assume now that $b<c$, so that $c-1\geq b$. Notice that if we multiply twice with $u!$, the inequality to prove becomes
        \[ (t+c-1) \cdots (t+c-u) \cdot \frac{b!}{(b-u)!}\ \preceq\ (t+b)\cdots (t+b-u+1)\cdot \frac{c!}{(c-u)!}\enspace .\]  
    which can be rewritten as
        \[\frac{(c-u)!}{c!} \cdot (t+c-1) \cdots (t+c-u)\ \preceq\  \frac{(b-u)!}{b!}\cdot (t+b)\cdots (t+b-u+1)\enspace .\]
    And this is equivalent to
        \[ \frac{c-u}{c} \cdot \left(\frac{t}{c-1} + 1\right) \cdots \left(\frac{t}{c-u} + 1\right)\ \preceq\ \left(\frac{t}{b} + 1\right) \cdots \left(\frac{t}{b-u+1} + 1\right)\enspace .\]
    And since $c-1\geq b$, and $\frac{c-u}{c}< 1$, the claim follows from property \eqref{rem:multiplication} by comparing the coefficients at each individual factor on the left with the corresponding factor on the right.
\end{proof}

We end this section with the proof of Theorem~\ref{thm:main} using the superadditivity of Proposition~\ref{prop:superadditivity} and the formulas of Theorem~\ref{thm:main1} and Corollary~\ref{cor:EhrHypersimplex}.
\begin{proof}[Proof of Theorem~\ref{thm:main}]
    Recall that the minimal matroid $T_{2,n}$ has exactly three hyperplanes, of cardinalities $1$, $1$, and $n-2$, respectively (cf. Example \ref{ex:minimalMatroid}). The uniform matroid $U_{2,n}$, on the other hand, has $n$ hyperplanes each of cardinality $1$. 
    
    Since we are under the hypothesis of $M$ being connected, we know that $M$ has at least $s\geq 3$ hyperplanes that partition the ground set. Assume that these hyperplanes have cardinalities $a_1,\ldots,a_s$. The sum of these numbers is $n$.
    
    By using Theorem \ref{thm:main1}, after cancelling $\binom{2t+n-1}{n-1}$ and multiplying by $-1$, the inequalities to prove read
    \begin{equation}\label{eqn:inequalities}
        \underbrace{P_{1,n}+\cdots + P_{1,n}}_{n\text{ summands}}\ \preceq\ \sum_{i=1}^s P_{a_i,n}\ \preceq\ P_{1,n} + P_{1,n} + P_{n-2,n}\enspace .
    \end{equation}
    The left inequality follows directly from the superadditivity in Proposition~\ref{prop:superadditivity}, since we may group the summands on the left into groups of sizes $a_1,\ldots,a_s$ and get the inequality with the expression in the middle. To prove the right inequality, we proceed by looking at inequality \eqref{iteration}. Recall that $s\geq 3$, so that we can assume $1\leq a_1\leq a_2\leq a_3$. By repeatedly applying \eqref{iteration} we get
    \begin{align*}
        P_{a_1,n} + P_{a_2,n} + P_{a_3,n}\ &\preceq\ P_{1,n} + P_{a_1+a_2-1,n} + P_{a_3,n}\\
        &\preceq\ P_{1,n} + P_{1,n} + P_{a_1+a_2+a_3-2,n}\enspace .
    \end{align*}
    Using the superadditivity again we arrive at
    \[ 
    \sum_{i=1}^s P_{a_i,s}\ \preceq\ P_{1,n} + P_{1,n} + P_{a_1+a_2+a_3-2,n} + \sum_{i=4}^s P_{a_i,s}\ \preceq\ P_{1,n} + P_{1,n} + P_{n-2,n}\enspace ,
    \]
    which completes the prove of the desired inequality. Corollary~\ref{cor:ehrMinimal} shows
    \[
    \ehr(\mathscr{P}(T_{2,n}),t)\ =\ \binom{t+n-1}{n-1} + (n-3)\binom{t+n-2}{n-1}
    \]
    which has only positive coefficients. In particular, it follows that all connected matroids of rank $2$ are Ehrhart positive.
    Moreover, the inequalities given in \eqref{eqn:inequalities} are strict for the coefficients of positive degree by Proposition~\ref{prop:superadditivity}. This proves that the coefficients of the Ehrhart polynomial of a connected rank $2$ matroid $M$ are strictly between those of the minimal and the uniform matroid unless the coefficient is the constant term or the matroid $M$ is either the minimal matroid $T_{2,n}$ or the uniform matroid $U_{2,n}$. Furthermore, recall that by Remark~\ref{rem:EhrPositivity} the Ehrhart polynomial of a disconnected rank~$2$ matroid is Ehrhart positive. This completes the proof.
\end{proof}

\section{\texorpdfstring{$h^*$}{h*}-polynomials}\label{sec:h_star}
\noindent In this section we prove Theorem~\ref{thm:main3}. We first deduce a formula for the $h^*$-polynomial of general rank $2$ matroids. We then use this formula to show that the $h^*$-polynomial is real-rooted in case of  sparse paving matroids.

\subsection{Matroids of rank two}
The following proposition provides a formula for the $h^*$-polynomials of general connected matroids of rank two. It is the counterpart of Theorem~\ref{thm:main1} for $h^*$-polynomials and, in fact, follows from it.
\begin{proposition}\label{main_but_for_h_star}
        Let $M$ be a connected matroid of rank $2$. Suppose that $M$ has exactly $s$ hyperplanes of sizes $a_1, \ldots, a_s$. Then $s\geq 3$ and we have
        \[ h^*(\mathscr{P}(M), x)\ =\ \sum_{j=0}^{\lfloor\frac{n}{2}\rfloor}\binom{n}{2j} x^j - \sum_{i=1}^{s} p^\ast_{a_i,n}(x)\enspace ,\]
    where
        \[ p^\ast_{a,n}(x) := \sum_{k=1}^a \sum_{j\geq 1} \binom{k}{j}\binom{n-k-1}{j-1}\, x^j
        \]
         for $1\leq a\leq n$.
\end{proposition}
\begin{proof}
By Theorem~\ref{thm:main1} it suffices to show that 
\begin{equation}\label{eq:katzman}
    \binom{2t+n-1}{n-1} \ = \ \sum_{j=0}^{\lfloor\frac{n}{2}\rfloor} \binom{n}{2j}\binom{t+n-1-j}{n-1}
\end{equation}
and
\begin{equation}\label{eq:pahstar}
     P_{a,n}(t)\ =\ \sum_{k=1}^a \sum_{j \geq 1} \binom{k}{j}\binom{n-k-1}{j}\binom{t+n-1-j}{n-1}  \, .
\end{equation}
for all $1\leq a\leq n$.

Equation~\eqref{eq:katzman} is a special case of~\cite[Theorem 2.5]{katzman-hypersimplices}. See also \cite[Corollary 8]{savage-schuster} for a combinatorial proof.

Equation~\eqref{eq:pahstar} follows from applying Sur\'anyi's identity (Lemma~\ref{lem:Suranyi}) twice. Lemma~\ref{lem:Suranyi} for $r=n-k$ and $s=k-1$ yields
    \begin{align*}
     \binom{t+n-k}{n-k}\binom{t+k-1}{k-1} =\sum_{j \geq 0} \binom{n-k}{j}\binom{k-1}{j}\binom{t+n-j-1}{n-1}\\
        = \sum_{j \geq 0}  \binom{n-k-1}{j} \binom{k-1}{j}\binom{t+n-j-1}{n-1}+\sum_{j \geq 1}  \binom{n-k-1}{j-1} \binom{k-1}{j}\binom{t+n-j-1}{n-1}
    \end{align*}
Setting $r=n-k-1$ and $s=k-1$ in Lemma~\ref{lem:Suranyi} we obtain
    \begin{align*}
    \binom{t+n-k-1}{n-k-1}\binom{t+k-1}{k-1} = \sum_{j\geq 0} \binom{n-k-1}{j}\binom{k-1}{j}\binom{t+n-j-2}{n-2}\enspace .
    \end{align*}
    Taking the difference of these two expressions and using Pascal's identity shows that $\binom{t+n-k-1}{n-k}\binom{t+k-1}{k-1}$ is equal to
    \begin{align*}
       \sum_{j \geq 0} \binom{n-k-1}{j}\binom{k-1}{j}\binom{t+n-j-2}{n-1} + \sum_{j \geq 1} \binom{n-k-1}{j-1}\binom{k-1}{j}\binom{t+n-j-1}{n-1}.
    \end{align*}
    After shifting the index of the first sum by one, this simplifies to 
    \[
 \sum_{j \geq 1} \binom{n-k-1}{j-1}\binom{k}{j}\binom{t+n-j-1}{n-1} \, .
    \]
Thus the polynomial $P_{a,n}(t)$ is equal to
\[
%P_{a,n}(t) \ = \ 
\sum_{k=1}^a \binom{t+n-k-1}{n-k}\binom{t+k-1}{k-1} \ = \ \sum_{k=1}^a \sum_{j \geq 1} \binom{k}{j}\binom{n-k-1}{j-1}\binom{t+n-1-j}{n-1}
\]
as desired.
\end{proof}

\subsection{Sparse paving matroids}
Now we turn our attention to the real-rootedness of the $h^*$-polynomial for sparse paving matroids of rank $2$. Recall that hyperplanes of such matroids have size at most $2$.  
In particular, Proposition \ref{main_but_for_h_star} simplifies in this case.

\begin{corollary}\label{h_star_for_sparse_paving}
    Let $M$ be a connected sparse paving matroid of rank $2$ on $n$ elements, and let $\lambda$ be the number of hyperplanes of size $2$ of $M$. Then
    \[ h^*(\mathscr{P}(M), x)\ =\ \sum_{j=0}^{\lfloor\frac{n}{2}\rfloor} \binom{n}{2j}\, x^j - (n+\lambda)\,x - \lambda\, (n-3)\,x^2\enspace .\]
\end{corollary}

\begin{proof}
    If the matroid $M$ of rank $2$ is sparse paving and has $\lambda$ hyperplanes of size $2$ the remaining hyperplanes must all be of size $1$. Moreover, as all hyperplanes partition the ground set the number of hyperplanes of size one must be $n-2\lambda$.
    The desired result follows from Proposition \ref{main_but_for_h_star} using these numbers.
\end{proof}

The general idea of our proof of Theorem~\ref{thm:main3} is to consider the substitution $x=-\tan^2 \theta$. This will then allow us to find explicit values for which the $h^*$-polynomial alternates in sign. We then obtain its real-rootedness from the Intermediate Value Theorem. 

A key ingredient of the proof is the following lemma.

%The general idea of parametrizing the negative real numbers via the substitution $x= -\tan^2\theta$ in the formula that we obtained in Corollary \ref{h_star_for_sparse_paving} allows us to find explicit values for which the polynomial alternates in sign which then leads to the real-rootedness of the original polynomial. Before getting to that, we establish a technical Lemma that will be a key part of the main proof.

\begin{lemma}\label{lem:bounds}
    For $a\in\{1,2\}$, all $n\geq 9$ and all $y\in (0,1]$
    \[
    -\frac{a}{n} \,<\, y^{\frac{n}{2}}\, p_{a,n}^\ast \left(\frac{y-1}{y}\right) \, < \, \frac{a}{n} \enspace .
    \]
\end{lemma}

\begin{proof}
    Let us consider the cases $a=1$ and $a=2$ separately.
    
    If $a=1$, the inequalities to prove reduce to
            \[ -\frac{1}{n}\ <\ y^{\frac{n}{2}-1}\,(y-1)\ <\ \frac{1}{n}\enspace.\]
        As we are working under the assumption that $y\in (0,1]$, the expression in the middle is always nonpositive, therefor it suffices to prove the left inequality. After multiplying by $-1$, this reads
            \[ y^{\frac{n-2}{2}}\,(1-y)\ <\ \frac{1}{n}\enspace .\]
        Consider the function $f(y) = y^{\frac{n-2}{2}}(1-y)= y^{\frac{n-2}{2}} - y^\frac{n}{2}$ with domain $(0,1]$. Differentiating the function $f$ we obtain that its maximum value is attained at the unique critical point, which is located at $\overline{y} = \frac{n-2}{n}=1-\frac{2}{n}$. This yields
        \begin{align*}
            f(y) \ \leq \ f(\overline{y}) \ = \ \frac{2}{n} \left(1-\frac{2}{n}\right)^{\frac{n-2}{2}}\enspace .
        \end{align*}
        The sequence $\left(1-\tfrac{2}{n}\right)^{\frac{n-2}{2}}$ is strictly decreasing and converges to $\exp(-1)\approx 0.36$; moreover, it is already less than $\frac{5}{12}\approx 0.42$ for all $n\geq 9$. In particular, for $n\geq 9$ we obtain
            \[y^{\frac{n}{2}-2}\,(1-y) \ =\ f(y) \ <\ \frac{2}{n} \cdot \frac{5}{12} \ =\ \frac{5}{6n} \ <\ \frac{1}{n}\enspace ,\]
        which yields the desired inequality.
        
        If $a=2$, the inequalities to show are
            \[ -\frac{2}{n} \ <\  u(y) \ <\ \frac{2}{n}\enspace ,\]
        where $u(y) = y^{\frac{n-4}{2}} (y-1)(ny-n+3)$. Let $f(y)$ be as in the case $a=1$ above, and let $g(y)$ be
        \[ g(y)\ =\ y^{\frac{n}{2} - 2} \,(y-1)\,((n-2)y-n+3)\enspace .\]
        Then $u(y) = g(y)-2f(y)$. Hence, as we already proved $0 < f(y) < \frac{5}{6n}$ for $y\in (0,1]$, it suffices to show that
        \[ -\frac{1}{3n} \ \leq\ g(y) \ \leq\ \frac{2}{n} \enspace .\]
        The function $g(y)$ vanishes on the boundary of the interval $[0,1]$; and, for $y\in (0,1)$ we have $g(y)< 0$ if and only if $ y> 1-\tfrac{1}{n-2}$. Furthermore, the derivative of $g$ is
        \[
        g'(y)\ =\ \frac{1}{2}\, y^{\frac{n}{2}-3}\,\left(n(n-2)y^2+(-2n^2+9n-10)y+(n-4)(n-3)\right)
        \]
        from which we obtain two critical points besides $0$  at
        %\[
        %1-\frac{5}{2n}\pm \frac{1}{n}\sqrt{\frac{17}{4}-\frac{4}{n-2}}\, .
        %\]
        %Let
        \[
        y_{\max} = 1-\frac{5}{2n}- \frac{1}{n}\sqrt{\frac{17}{4}-\frac{4}{n-2}}
        \;\; \text{ and } \;\;
        y_{\min } = 1-\frac{5}{2n}+ \frac{1}{n}\sqrt{\frac{17}{4}-\frac{4}{n-2}}\enspace .
        \]
        If $n > 18$, then $\frac{4}{n-2}< \frac{1}{4}$ and therefore 
        \[
        c_1:=1-\frac{1}{2n}\ <\ y_{\min}\ <\ 1-\frac{5-\sqrt{17}}{2n}=:c_2\ <\ 1
        \]
        and 
        \[
        d_1:=1-\frac{5+\sqrt{17}}{2n}\ <\ y_{\max}\ <\ 1-\frac{9}{2n}=:d_2\ <\ 1 \enspace .
        \]
        In particular, $y_{\min}>c_1>1-\tfrac{1}{n-2}>d_2>y_{\max}$ and hence $g(y_{\min}) < 0 < g(y_{\max})$. We conclude that $g(y)$ restricted to the interval $[0,1]$ attains its minimum at $y_{\min}$ and its maximum at $y_{\max}$. Therefore it suffices to prove that $g(y_{\max})\leq \frac{2}{n}$ and $g(y_{\min})\geq -\frac{1}{3n}$.
        
        With the bounds given above, and assuming again that $n>18$, we obtain
        \begin{eqnarray*}
        g(y_{\min})&>&1^{\frac{n}{2}-2}(1-c_1)((2-n)c_2+n-3)\\
        &=& \frac{1}{2n}\left(\frac{3-\sqrt{17}}{2}-\frac{5-\sqrt{17}}{n}\right)\\
        &>& \frac{1}{2n}\left(\frac{3-\sqrt{17}}{2}-\frac{5-\sqrt{17}}{18}\right)\\
        &=& \frac{1}{n}\cdot \frac{11-4\sqrt{17}}{18} \\
        &>& -\frac{1}{3n} \, .
        \end{eqnarray*}
        If furthermore $n\geq 27$ then
        \begin{eqnarray*}
        g(y_{\max})&<&d_2^{\frac{n}{2}-2}(1-d_1)((2-n)d_1+n-3)\\
        &=&\left(1-\frac{9}{2n}\right)^{\frac{n}{2}-2}  \frac{5+\sqrt{17}}{2n}\cdot\frac{3n-10+(n-2)\sqrt{17}}{2n}\\
        &=&\left(1-\frac{9}{2n}\right)^{\frac{n}{2}-2} \frac{8n-21+(2n-5)\sqrt{17}}{n^2}\\
        &<&\left(1-\frac{9}{2n}\right)^{\frac{n}{2}-2} \frac{8+2\sqrt{17}}{n}\\
        &<&\frac{2}{n}
        \end{eqnarray*}
        where in the last step we used that $\left(1-\frac{9}{2n}\right)^{\frac{n}{2}-2}$ is monotone decreasing and smaller than $\sqrt{17}-4$ for $n\geq 27$. For the remaining cases $9\leq n  < 27$ we used the computer to verify the upper and lower bounds for $g(y)$.
\end{proof}

The second key ingredient of our proof of Theorem~\ref{thm:main3} is the following  trigonometric identity.
%expression for the image of $\sum_{j=0}^{\lfloor\frac{n}{2}\rfloor} \binom{n}{2j} x^{j}$ under the substitution $x=-\tan^2 \theta$.

\begin{lemma}\label{lem:transHstar}
    For all $\theta \in [0,\pi /2)$ and $n\geq 0$,
    \[
    \sum_{j=0}^{\lfloor\frac{n}{2}\rfloor} \binom{n}{2j} \left(-\tan^2\theta\right)^j\ =\ \frac{\cos n\theta}{\cos^n \theta} \, .
    \]
\end{lemma}
\begin{proof}
    As a consequence of the binomial theorem, for all $x\geq 0$
    \[
    \sum _{j=0}^{\lfloor \frac{n}{2}\rfloor} (-1)^j \binom{n}{2j} x^j\ =\ \frac{\left(1+\ii \sqrt{ x }\right)^n+\left(1-\ii\sqrt{x}\right)^n}{2}
    \]
    where $\ii$ denotes the imaginary unit. Substituting $x= \tan^2 \theta$ and using $\tan \theta = \tfrac{\sin \theta}{\cos \theta}$ then  yields
    \begin{eqnarray*}
    \sum_{j=0}^{\lfloor\frac{n}{2}\rfloor} (-1)^j\binom{n}{2j}\left(\tan^2\theta \right)^j&=& \frac{(1+\ii\cdot \tan\theta)^n + (1-\ii\cdot\tan\theta)^n}{2}\\
    &=& \frac{(\cos \theta + \ii\cdot \sin \theta)^n + (\cos \theta - \ii\cdot \sin \theta)^n}{2\cos^n \theta}\\ 
    &=&\frac{e^{\ii\, n\,\theta} + e^{-\ii\, n\,\theta}}{2\cos^n \theta}\\
    &=&\frac{\cos n\theta}{\cos^n\theta}
    \end{eqnarray*}
    for all $0\leq \theta < \pi /2$ as desired.
\end{proof}

We are now prepared to put the pieces together and prove Theorem~\ref{thm:main3}.
\begin{proof}[Proof of Theorem~\ref{thm:main3}] We discuss the case of  disconnected paving matroids before we turn our attention to the connected case.  

    By Lemma~\ref{lem:disconnected}, a loopless matroid of rank $2$ with three or more hyperplanes must be connected.
    In particular, a disconnected rank $2$ sparse paving matroid can have at most two hyperplanes of sizes one or two that partition its ground set. 
    
    There are exactly three disconnected sparse paving matroids of rank $2$, namely $U_{1,1}\oplus U_{1,1}$, $U_{1,1}\oplus U_{1,2}$ and $U_{1,2}\oplus U_{1,2}$. 
    By Example~\ref{ex:hstarDisconnected}, their $h^*$-polynomials are $1$, $1$, and $1+x$ respectively. Also the $h^*$-polynomials of the uniform matroid $U_{2,3}$ and the minimal matroid $T_{2,4}$ are $1$ and $1+x$.
    These are exactly the polynomials listed first in Theorem~\ref{thm:main3}.
    
    Now we are left with the case of connected sparse paving matroids. Corollary \ref{h_star_for_sparse_paving} gives us the formula
    \[ h^*(\mathscr{P}(M), x) \ =\  \sum_{j=0}^{\lfloor\frac{n}{2}\rfloor} \binom{n}{2j} x^j - (n-2\lambda)\, p^*_{1,n}(x) - \lambda\, p^*_{2,n}(x)\]
    for their $h^*$-polynomials which are of degree $\big\lfloor\frac{n}{2}\big\rfloor$ whenever the matroid $M$ is neither $U_{2,3}$ nor $T_{2,4}$.
    Applying Lemma~\ref{lem:transHstar}, we obtain
    \begin{align*}
        h^*(\mathscr{P}(M), -\tan^2\theta ) &=  \frac{\cos n\theta} {\cos^n\theta}  - \lambda\, p^*_{2,n}(-\tan^2\theta) - (n-2\lambda)\, p^*_{1,n}(-\tan^2\theta)\\
        &=  \frac{\cos n\theta  - \lambda\, \cos^n \theta \, p^*_{2,n}(-\tan^2\theta) - (n-2\lambda)\, \cos^n\theta\, p^*_{1,n}(-\tan^2\theta)}{\cos^n\theta}\, .
    \end{align*}
  Let us denote the numerator of this fraction by $q(\theta)$. Observe that using the substitution $y=\cos^2\theta$, we obtain
  %\[ q(\theta) \ =\ \cos n\theta - \lambda y^\frac{n}{2} p^*_{2,n}\left(\tfrac{y-1}{y}\right) - (n-2\lambda) y^{\frac{n}{2}}p^*_{1,n}\left(\tfrac{y-1}{y}\right).\] 
  \[ \cos n\theta - q(\theta) \ =\  \lambda\, y^\frac{n}{2}\, p^*_{2,n}\left(\tfrac{y-1}{y}\right) + (n-2\lambda)\, y^{\frac{n}{2}}\, p^*_{1,n}\left(\tfrac{y-1}{y}\right)\enspace.\] 
  In particular, as $y\in (0,1]$, we can apply Lemma \ref{lem:bounds} to derive the bounds
  %\[ \cos n\theta - \lambda\cdot\tfrac{2}{n} - (n-2\lambda) \tfrac{1}{n} < f(\theta) < \cos n\theta + \lambda\cdot\tfrac{2}{n} + (n-2\lambda) \tfrac{1}{n}, \]
   \[ -1 \ =\ - \lambda\, \tfrac{2}{n} - (n-2\lambda)\, \tfrac{1}{n}\ <\ \cos n\theta - q(\theta)\ <\ \lambda\, \tfrac{2}{n} + (n-2\lambda)\, \tfrac{1}{n}\ =\ 1 \enspace , \]
  for all $\theta\in [0,\frac{\pi}{2})$, which is equivalent to
  \begin{equation}\label{inequality-for-cos} \cos n\theta - 1\ <\ q(\theta)\ <\ \cos n\theta + 1.
  \end{equation}
  Let us consider the points $\theta_j := \frac{j\pi}{n}$ for $j=0,\ldots,\big\lfloor\frac{n}{2}\big\rfloor$.
 The inequalities \eqref{inequality-for-cos} imply that $q(\theta_j)>\cos j\pi -1 = 0$ whenever $j$ is even, and $q(\theta_j) < \cos j\pi +1 = 0$ whenever $j$ is odd.
  In particular, since thd function $q$ is continuous, by the Intermediate Value Theorem, we obtain that $q$ has a zero in each open interval $(\theta_j, \theta_{j+1})\subseteq[0,\frac{\pi}{2})$ for $j=0,\ldots,\big\lfloor\frac{n}{2}\big\rfloor-1$. Since
    \[ h^*(\mathscr{P}(M), -\tan^2\theta)\ =\ \frac{q(\theta)}{\cos^n\theta}\enspace ,\]
    and $\cos^n\theta$ is strictly positive for $\theta\in [0,\frac{\pi}{2})$, we obtain that $h^*(\mathscr{P}(M),-\tan^2\theta)$ has $\big\lfloor\frac{n}{2}\big\rfloor$ zeros in the interval $[0,\frac{\pi}{2})$. Since $\theta \mapsto -\tan^2 \theta$ is injective on the interval $[0,\tfrac{\pi}{2})$ this implies that $h^*(\mathscr{P},x)$ has $\big\lfloor\frac{n}{2}\big\rfloor$ distinct negative real roots. Since this number agrees with the degree of the $h^*$-polynomial, we conclude that it is real-rooted, as claimed.
\end{proof}

\section{Open Questions}\label{sec:Outlook}
\noindent

We conclude by discussing open questions and possible directions for future research on Ehrhart polynomials of matroid polytopes.

\subsection{\texorpdfstring{$h^*$}{h*}-polytopes of rank two matroids}

In~\cite{ferroni2020ehrhart} it was conjectured that $h^*$-polynomials of matroid base polytopes are real-rooted for all matroids. In Theorem~\ref{thm:main3} we have answered the question in the affirmative for sparse paving matroids of rank $2$. It would be interesting to see if our approach can be extended to show real-rootedness for all matroids of rank $2$.

A first step towards proving this conjecture could be to show that
for every $a\geq 1$ there is an $N$ such that for $n>N$ and  $y\in (0,1]$
\begin{align}\label{eq:strongBounds}
    -\frac{a}{n} \ <\  y^{\frac{n}{2}}\, p_{a,n}^\ast \left(\frac{y-1}{y}\right) \ < \ \frac{a}{n} \enspace .
\end{align}
The same arguments as in the proof of Theorem~\ref{thm:main3} then would show that matroids with hyperplanes of size at most $a$ and sufficiently large ground set have matroid polytopes whose $h^*$-polynomial is real-rooted. However, notice that the $h^*$-polynomial has degree lower than $\big\lfloor\frac{n}{2}\big\rfloor$ whenever there is a hyperplane of size $\big\lfloor\frac{n}{2}\big\rfloor$ or larger. The information about the degree of the $h^*$-polynomial is critical to our approach to construct the correct number of points such that the corresponding evaluations of the $h^*$-polynomial alternate in sign in order to be able to use the Intermediate Value Theorem. It would be interesting to see if and how the approach can be modified to lead to a proof of the real-rootedness of arbitrary rank $2$ matroids and possibly to higher rank matroids. 

\subsection{Positroids}
In \cite{postnikov} Postnikov introduced positroids. These are the matroids that are representable over the real numbers by a matrix whose maximal minors are all nonnegative. Alternatively, positroids can be characterized geometrically by the following property of their base polytopes; see \cite[Proposition 5.6]{ardila-rincon-williams} or \cite[Theorem 2.1]{lam-postnikov}.
\begin{theorem}
    A matroid $M$ is a positroid if and only if the polytope $\mathscr{P}(M)$ can be written using inequalities of the form
        \[\alpha_{i,j}\ \leq\ x_i + x_{i+1} + \cdots + x_j\ \leq\ \beta_{i,j}\]
    for some $\alpha_{i,j},\beta_{i,j}\in\ZZ$.
\end{theorem}

Recall that as a consequence of Equation \eqref{ineq:hyperplanes} we can see that, up to a relabelling of the ground set, any rank $2$ matroid is a positroid. More precisely, if we assume that a rank $2$ matroid on $n$ elements has its elements labeled such that all the hyperplanes $H_1,\ldots, H_s$ consist of consecutive numbers, i.e., each hyperplane is of the form $H = \{i,i+1,\ldots,j\}$, then we obtain a description of $\mathscr{P}(M)$ that reveals that it is a positroid. Observe that here we are leveraging the fact that the hyperplanes of a rank $2$ matroid form a partition of the ground set.

\begin{corollary}
    All rank $2$ matroids are isomorphic to positroids of rank $2$.
\end{corollary}

Since we have proved that polytopes of rank $2$ matroids indeed are Ehrhart positive, and the first author showed that this also holds true for uniform matroids and minimal matroids in any rank \cite{ferroni-hypersimplices,ferroni2020ehrhart}, we believe that this might be a phenomenon that extends to all positroids.

\begin{conj}
    Positroids are Ehrhart positive.
\end{conj}

It is worth mentioning that positroids encompass (up to isomorphism) another well studied family of matroids, namely lattice path matroids \cite[Lemma 23]{oh}. Lattice path matroids contain all uniform and minimal matroids, but not all matroids of rank $2$. On the other hand, the fact that positroids are representable matroids over the real field gives rise to questions on Ehrhart positivity for representable matroids.  \\

\noindent\textbf{Acknowledgements:} LF is supported by a Marie Sk{\l}odowska-Curie PhD fellowship as part of the program INdAM-DP-COFUND-2015, Grant Number 713485. KJ is supported by the Wallenberg AI, Autonomous Systems and Software Program
funded by the Knut and Alice Wallenberg Foundation, as well as Swedish
Research Council grant 2018-03968 and the G\"oran Gustafsson Foundation.
BS is supported by the Knut and Alice Wallenberg Foundation.

\bibliographystyle{abbrv}
\bibliography{bibliography}

\end{document}